\documentclass[reqno]{amsart}

\usepackage{hyperref}
\usepackage{appendix}
\usepackage{float}
\usepackage{amsmath}
\usepackage{graphicx}
\usepackage{latexsym}
\usepackage{amsfonts} 
\usepackage{amssymb}
\usepackage{verbatim}
\usepackage{color}
\usepackage{todonotes}
\usepackage{cleveref}
\theoremstyle{plain}
\newtheorem{theorem}{Theorem}
\newtheorem{corollary}[theorem]{Corollary}
\newtheorem{lemma}[theorem]{Lemma}
\newtheorem{proposition}[theorem]{Proposition}
\theoremstyle{definition}
\newtheorem{assumption}{Assumption}

\newtheorem{remark}[theorem]{Remark}
\newtheorem*{remark*}{Remark}

\newcommand{\R}{\mathbb R}
\newcommand{\Z}{\mathbb Z}

\renewcommand{\P}{\mathbf P}

\newcommand{\E}{\mathbf E}

\newcommand\eps\varepsilon
\newcommand{\N}{{\mathbb{N}}}

\def\wE{\widehat{\E}}
\def\wP{\widehat{\P}}

\begin{document}
\title[Co-existence of branching populations in random environment]
{Co-existence of branching populations in random environment} 
\thanks{Funded by the Deutsche Forschungsgemeinschaft (DFG, German Research Foundation) –
Project-ID 317210226 – SFB 1283.}
\author[Elizarov]{Nikita Elizarov}
\address{Faculty of Mathematics, Bielefeld University, Germany}
\email{nelizarov@math.uni-bielefeld.de}

\author[Wachtel]{Vitali Wachtel}
\address{Faculty of Mathematics, Bielefeld University, Germany}
\email{wachtel@math.uni-bielefeld.de}

\begin{abstract}
        In this paper we consider two branching processes living in a joint random environment. Assuming that both 
        processes are critical we address the following question: What is the probability that both populations 
        survive up to a large time $n$? We show that this probability decays as $n^{-\theta}$ with $\theta>0$ which is 
        determined by the random environment. Furthermore, we prove the corresponding conditional limit theorem. One of the main ingredients in the proof is a qualitative bound for the entropic repulsion for two-dimensional random walks conditioned to stay in the positive quadrant. We believe that this bound is also of independent interest.
\end{abstract}


\maketitle

\section{Introduction, main results and discussion}
In this paper we consider a two-dimensional {\it purely decomposable} branching process in a random environment. This model has been introduced by Vatutin and Liu~\cite{VL13}. Let us give a formal description of that process. Let $\mathcal{M}$ denote the space of probability measures on non-negative integers $\mathbb{N}_0$. Let $\mathbf{Q}=(Q_1,Q_2)$
be a random vector whose coordinates take values in $\mathcal{M}$. An environment is then modeled by a sequence $\mathbf{\Pi}=\{\mathbf{Q}(n)\}_{n\ge1}$ of independent copies of $\mathbf{Q}$. We shall consider a time-homogeneous Markov chain $Z(n)=(Z_1(n),Z_2(n)),\ n\ge0$ on $\mathbb{N}_0\times \mathbb{N}_0$, whose transition rules are defined as follows. Given $\mathbf{Q}(n)=(q_1,q_2)$ and $Z(n)=(z_1,z_2)$, the law of $Z(n+1)$ coincides 
with the distribution of the vector
$$
\left(\sum_{j=1}^{z_1}\xi^{(n)}_1(j),\sum_{j=1}^{z_2}\xi^{(n)}_2(j)\right),
$$
where $\{\xi^{(n)}_i(j)\}$ are i.i.d. random variables with the law $q_i$, $i=1,2$. It is immediate from the definition that, conditionally on $\mathbf{\Pi}$, the processes $Z_1(n)$ and $Z_2(n)$ are independent. Thus, the corresponding populations interact via joint environment only. 

Let $\{X(n)\}$ be a sequence of two-dimensional vectors defined by 
\begin{align}
\label{eq:X-def}
X_i(n)=\log\E[\xi^{(n)}_i(1)|\mathbf{Q}(n)],\quad i=1,2;\ n\ge1.
\end{align}
Clearly, that vectors are independent and identically distributed.
We shall always assume that 
\begin{align}
\label{eq:zeromean}
\E X_1(n)=\E X_2(n)=0
\end{align}
and 
\begin{align}
\label{eq:unitvar}
\E X^2_1(n)=\E X^2_2(n)=1.
\end{align}
Under these moment restrictions, both processes $Z_1(n)$ and $Z_2(n)$ are critical and there exist positive constants $\alpha_i$ such that 
\begin{align}
\label{eq:nonex-1dim}
\P(Z_i(n)>0)\sim \frac{\alpha_i}{\sqrt{n}}\quad\text{as }n\to\infty,
\end{align}
see, for example, Chapter 5 in the book~\cite{GKbook} by Kersting and Vatutin.
This yields the following bounds for the non-extinction probability for 
$Z(n)$:
\begin{align*}
&\liminf_{n\to\infty}\sqrt{n}\P(Z(n)\neq0)\ge\max\{\alpha_1,\alpha_2\},\\
&\limsup_{n\to\infty}\sqrt{n}\P(Z(n)\neq0)\le \alpha_1+\alpha_2.
\end{align*}
In particular, the non-extinction probability is of the same order as for processes with one type of particles. This is a quite standard situation for multi-type branching processes in random environments, see Le Page et al.~\cite{PPP18} and references there. This 'one-dimensionality' of non-extinction probabilities can be underlined by the simple observation
that
$$
\{Z(n)\neq0\}=\{Z_1(n)+Z_2(n)>0\}.
$$
Thus, non-extinction of $Z(n)$ is equivalent to positivity of the one-dimensional sequence $Z_1(n)+Z_2(n)$. 

In the present paper we address a question on the co-existence of both types up to time $n$. More precisely, we want to understand the behavior of the probability $\P(Z_1(n)>0,Z_2(n)>0)$. It is rather clear that the moment assumptions \eqref{eq:zeromean}, \eqref{eq:unitvar} are not sufficient for this problem, since they do not describe interactions between types. It turns out, and this is not surprising, that the interactions of populations are described by the covariance 
\begin{align}
\label{eq:covar}
\varrho:=\E[X_1(n)X_2(n)].
\end{align}
We shall also assume that all possible offspring distributions are geometric ones. More precisely, we assume that every realization $q$
of the vector $Q$ is of the form 
\begin{align}
\label{eq:geom}
q_i(\{j\})=p_i(1-p_i)^{j},\quad j\in\N_0,\ i=1,2. 
\end{align}
The corresponding value of the vector $X$ equals
$\left(\log((1-p_1)/p_1), \log((1-p_2)/p_2)\right)$.
This assumption is not crucial and is made to avoid some additional purely technical considerations and to reduce the volume of the paper. Assumption \eqref{eq:geom} can be removed by the method suggested in Geiger and Kersting~\cite{GK02}.
\begin{theorem}
\label{thm:coexistence}
Assume that \eqref{eq:zeromean}, \eqref{eq:unitvar} and \eqref{eq:geom}
hold and that $|\varrho|<1$. Set 
$$
\theta(\varrho)=\frac{\pi}{2\arccos(-\varrho)}.
$$
Assume also that $\E|X(n)|^{2\theta}$ and $\E|X(n)|^2\log(1+|X(n)|)$ are finite.
Then for every starting point $z$ there exists a positive constant $A=A(z)$ such that 
\begin{align}
\label{eq:coex}
\P_z(Z_1(n)>0,Z_2(n)>0)\sim An^{-\theta(\varrho)}\quad\text{as }n\to\infty. 
\end{align}
\end{theorem}
Our assumption $|\varrho|<1$ implies that $\theta(\varrho)\in(1/2,\infty)$. Thus, combining \eqref{eq:coex} with \eqref{eq:nonex-1dim}, we conclude that
$$
\P(Z(n)\neq0)\sim\frac{\alpha_1+\alpha_2}{\sqrt{n}}.
$$

Let us now comment on boundary cases $\varrho=\pm1$.

If $\varrho=1$ then $X_1(n)=X_2(n)$ for all $n$. This means that every environment is either good or bad for both populations. In this case one can easily infer that the co-existence probability is asymptotically equivalent to $Cn^{-1/2}$ with some positive constant $C$. 

If $\varrho=-1$ then $X_1(n)=-X_2(n)$ for all $n$. This situation is opposite to the case $\varrho=1$: if a particular environment is good for the first population then it is bad for the second one, and vice versa. Because of this effect Vatutin and Liu~\cite{VL13} has called this environment {\it asynchronous}. Their paper studies quenched distribution of the process conditioned on co-existence and observe appearance of bottlenecks. For their purposes they do not need to know the co-existence probability. Since $\theta(\varrho)$ is monotone decreasing and 
$\lim_{\varrho\to-1}\theta(\varrho)=\infty$, it is natural to expect that
$\P(Z_1(n)>0,Z_2(n)>0)$ decreases faster than any power. 
More precisely, using large deviation arguments, one can show that 
$$
\log\P(Z_1(n)>0,Z_2(n)>0)\sim -cn^{1/3}
$$
with some positive constant $c>0$. But it is absolutely unclear to us, how to get exact asymptotics for the co-existence probability
in the case $\varrho=-1$.

The proof of Theorem~\ref{thm:coexistence} uses the standard for branching processes in random environments relation between processes and so-called
{\it associated random walks.} We set 
$$
S(n)=X(1)+X(2)+\ldots+X(n),\quad n\ge 1.
$$
This sequence describes the behaviour of conditional expectations of the process $Z(n)$: for every $n\ge1$ one has
$$
\E\left[Z_i(n)|\mathbf{\Pi}\right]=e^{S_i(n)},\quad i=1,2.
$$
Similar to the case of branching processes with one type of particles, it is natural to expect that the co-existence occurs for environments, where
conditional expectations of both coordinates remain separated from zero. This means, roughly speaking, that the walk $S(k)$, $k\ge1$ should stay in the positive quadrant $\R_+^2$ up to time $n$. This problem has been studied by Denisov and Wachtel in \cite{DW15,DW24}. The result of Theorem~\ref{thm:coexistence} can be reformulated then as follows:
$$
\P(Z_1(n)>0,Z_2(n)>0)\sim c\P(S(k)\in\R_+^2\text{ for all }k\le n).
$$
Here we assume for simplicity that $\P(X(n)\in\R_+^2)$ is positive. Otherwise one has to replace $S(k)$ by $(R,R)+S(k)$ with sufficiently large $R$ and with constant $c$ depending on $R$.

The most difficult step in the proof of Theorem~\ref{thm:coexistence} is the Proposition~\ref{prop:repulsion} below, where we give a quantitative bound for the so-called entropic repulsion for the random walk $\{S(n)\}$ conditioned to stay in $\R_+^2$ at all times.
This result allows us to conclude that the constant $A$ in \eqref{eq:coex} is strictly positive.

We next formulate a limit theorem for logarithmically scaled sizes of populations conditioned on co-existence. On the set $E(n):=\{Z_1(n)>0,Z_2(n)>0\}$ we define
$$
z^{(n)}(t)=\left(\frac{\log Z_1(nt)}{\sqrt{n}},\frac{\log Z_1(nt)}{\sqrt{n}}\right),\quad t\in[0,1].
$$
\begin{theorem}
\label{thm:cond.limit}
Suppose all the conditions of Theorem~\ref{thm:coexistence} hold. Then the sequence $z^{(n)}$ conditioned on $E(n)$ converges in distribution on $D[0,1]$ with uniform metric. (The limiting process can be called Brownian meander in $\R_+^2$ and is formally defined in the proof below.)
\end{theorem}
Having asymptotics for the probability of $E(n)$, the proof of this limit theorem is not very difficult and follows the standard path for branching processes in random environments, see Section 5.7 in\cite{GKbook}.

\section{Associated random walk in the positive quadrant.}
In this section we collect needed in the proof of Theorem~\ref{thm:coexistence} results on the associated random walk 
$S(n)$. 

We start by summarizing known in the literature results for random walks confined to cones.
Set
$$
T=
\left(
\begin{array}{cc}
\frac{1}{\sqrt{1-\varrho^2}}, &\frac{-\varrho}{\sqrt{1-\rho^2}}\\
0, &1
\end{array}
\right)
$$
It is easy to see that the vectors $Y(n)=TX(n)$ have then uncorrelated coordinates with unite variances. This linear transformation changes also the cone we are looking at: the image  $T(\R_+^2)$ of the positive quadrant has opening $\arccos(-\rho)\in(0,\pi)$. In particular, the cone $T(\R_+^2)$ is convex.

The linear transformation with matrix $T$ allows us to use all the results from \cite{DW15,DW24} with $p=\frac{\pi}{\arccos(-\varrho)}$ in the cone $T(R_+^2)$. Applying then the inverse transformation, we infer that all the results are valid for the walk $S(n)$ confined to the positive quadrant.  

Set
$$
\tau_x:=\inf\{n\ge1:x+S(n)\notin\R_+^2\},\quad x\in\R_+^2.
$$
Then, according to Theorem 3 in \cite{DW24},
\begin{align}
\label{eq:tau-tail}
\P(\tau_x>n)\sim\varkappa V(x)n^{-p/2}\quad\text{as }n\to\infty
\end{align}
for every $x\in\R_+^2$. The function $V$ is harmonic for the walk $S(n)$ killed at leaving the cone $\R_+^2$, that is,
\begin{align}
\label{eq:V-harm}
V(x)=\E[V(x+S(1));\tau_x>1].
\end{align}
Note also that
$$
\theta=\frac{p}{2}.
$$
In the proofs we shall use $p$ instead of $\theta$, since this notation is standard for random walks in cones.

We next define a process, which will appear as the limit in Theorem~\ref{thm:cond.limit}. 
We are going to use the results from \cite{DuW20}, where random walks with uncorrelated co-ordinates have been considered. We may apply Theorem 2 from that paper to the walk $\{TS(n)\}$ conditioned to stay in the cone $T(\R_+^2)$. This result states that 
$\frac{x+TS(nt)}{\sqrt{n}}$ conditioned to stay in $T(\R_+^2)$ converges weakly on $D[0,1]$ with the uniform metric towards the Brownian meander $M_{T(\R_+^2)}$; this process is constructed in Subsection~2.1 of~\cite{DuW20}. Applying the linear transformation $T^{-1}$, we conclude that the process 
$\frac{x+S(nt)}{\sqrt{n}}$, $t\in[0,1]$ conditioned on $\tau_x>n$ converges weakly on $D[0,1]$ towards the process
$M_\varrho(t)=T^{-1}M_{T(\R_+^2)}$.

The harmonicity of $V$ allows one to perform Doob's $h$-transform and to introduce new measures $\widehat{\P}_x$, $x$ is such that $V(x)>0$, characterized by equalities
\begin{align}
\label{eq:Doob-transform}
\widehat{\P}_x(A)=\frac{1}{V(x)}\E[V(x+S(n)){\rm 1}_A;\tau_x>n]
\end{align}
for every $A\in\sigma(S(1),S(2),\ldots,S(n))$ and every $n\ge1$.

One of the effects of this change of measure is the so-called entropic repulsion: ${\rm dist}(x+S(n),\partial\R_+^2)$ goes to infinity in $\widehat{\P}_x$-probability. Our next result gives a certain quantitative bound for this effect.
\begin{proposition}
\label{prop:repulsion}
Assume that all moment assumptions in Theorem~\ref{thm:coexistence}
hold. Then
$$
\widehat{\P}_x(x_i+S_i(n)\le \log^2 n\text{ i.o.})=0,\quad i=1,2
$$
for every $x\in\R_+^2$ such that $V(x)>0$. 
\end{proposition}
The remaining part of the section is devoted to the proof of this statement.
\subsection{Auxiliary estimates for the walk $S(n)$ in the positive quadrant.}
\begin{lemma}\label{small_V}
The exist a compact set $D$ and a constant $\delta>0$
such that
\begin{equation}
\label{eq:V-positiv}
V(z)\ge\delta
\end{equation}
and 
\begin{equation}
\label{eq:V>u}
V(z)\ge\delta u(z)
\end{equation}
for all $z\in\R_+^2\setminus D$.
\end{lemma}
\begin{proof}
        According to Theorem 2 in ~\cite{DW24}, 
        $$
        V(x) \sim u(x)\quad\text{as }dist(x,\partial\R_+^2)\to\infty.
        $$ 
        Thus, there exists $R > 0$ such that $\frac{1}{2}u(x)\leqslant V(x)$ for all $x\in\R_+^2$ such that
        $dist(x, \partial\R_+^2) \geqslant R.$  

        Next, let us notice that the assumption $\E X = 0$ implies that there exist $a, b > 0$ such that
        \begin{align*}
                \P\bigl(X_1 > a, X_2 > -b\bigr) =: q_1 > 0,
                \\
                \P\bigl(X_1 > -b, X_2 > a\bigr) =: q_2 > 0.
        \end{align*}
        We set $q := \min\{q_1, q_2\}$ and $n_0 := \frac{R}{a}.$

        Consider following sets:
        \begin{align*}
                B_1 := \{x\in\R_+^2 :\:\: x_2 > n_0 b + R,\:\: 0 < x_1 < R\},
                \\
                B_2 := \{x\in\R_+^2:\:\: x_1 > n_0 b + R,\:\: 0 < x_2 < R\}. 
        \end{align*}
        Set $x_0:=(1,1)$.
        For $x\in B_1,$ one has
        \begin{align*}
                \P\bigl(x + S(n_0)\in R x_0 + \R_+^2,\tau_x > n_0\bigr)&\geqslant\prod\limits_{k = 1}^{n_0}
                \P\bigl(X_1(k) > a, X_2(k) > -b\bigr)\\
                &\geqslant q_1^{n_0}\geqslant q^{n_0}.
        \end{align*}
        Similarly,
        \begin{align*}
                \P\bigl(x + S(n_0)\in R x_0 + \R_+^2,\tau_x > n_0\bigr)\geqslant p^{n_0}.
        \end{align*}
        The harmonicity of $V(x)$ stated in \eqref{eq:V-harm} implies that
        \begin{align*}
                V(x) = \E\bigl[V\bigl(x + S(n)\bigr),\:\:\tau_x > n\bigr]
        \end{align*}
        for all $n\ge1$ and all $x\in\R_+^2$.
        Thus, for all $x\in B_1\cup B_2$ we have
        \begin{align*}
                V(x)
                &\geqslant \P\bigl(x + S(n_0)\in Rx_0 + \R_+^2,\tau_x > n_0\bigr)V(Rx_0)\\ 
                &\geqslant q^{n_0}\frac{1}{2}u(Rx_0) > 0.
        \end{align*}
        For all $x\in Rx_0 + \R_+^2$ one also has $V(x) \geqslant V(Rx_0) \geqslant \frac{1}{2}u(Rx_0)$. Therefore, \eqref{eq:V-positiv} holds with $\delta=q^{n_0}\frac{1}{2}u(Rx_0)$ and 
        $D= \{x:\:\:0\leqslant x_1, x_2\leqslant R\}$.

        To prove the second claim we consider the sets
        \begin{align*}
                B'_1 := \{x\in\R_+^2 :\:\: x_2 > 2n_0 b + R,\:\: 0 < x_1 < R\},
                \\
                B'_2 := \{x\in\R_+^2 :\:\: x_1 > 2n_0 b + R,\:\: 0 < x_2 < R\}.
        \end{align*}
        Similarly to the proof of \eqref{eq:V-positiv}, we have
        \begin{align*}
            \P\bigl(x + S(n_0) \in x + (n_0 a, -n_0 b) + \R_+^2\bigr) \geqslant q_1^{n_0},\quad x\in B_1'
        \end{align*}
        and
        \begin{align*}
            \P\bigl(x + S(n_0) \in x + (-n_0 b, n_0 a) + \R_+^2\bigr) \geqslant q_2^{n_0},\quad x\in B_2'.
        \end{align*}
        Thus, for all $x\in B_1'$ we have
        \begin{align*}
                V(x)&\geqslant \P\bigl(x + S(n_0)\in x + (n_0 a, -n_0 b) + \R_+^2,\tau_x > n_0\bigr)\cdot 
                V\bigl(x + (n_0 a, -n_0 b)\bigr)
                \\
                &\geqslant q^{n_0}V\bigl(x + (n_0 a, -n_0 b)\bigr) . 
        \end{align*}
        As $x + (n_0 a, -n_0 b) + \R_+^2\subset x_0 R + \R_+^2$ one has
        \begin{align*}
                V(x)\geqslant q^{n_0}\cdot V\bigl(x + (n_0 a, -n_0 b)\bigr)\geqslant \frac{q^{n_0}}{2}
                u\bigl(x_1 + n_0 a, x_2 - n_0 b\bigr)
        \end{align*}
        As $x_2 > 2n_0 b + R$ we have $x_2 - n_0 b \geqslant \frac{1}{2}x_2,$ as well as $x_1 + n_0 a\geqslant 
        \frac{1}{2}x_1.$ Thus,
        \begin{align*}
                V(x)\geqslant \frac{q^{n_0}}{2}u\left(\frac{x}{2}\right)\geqslant Cu(x),\quad x\in B_1'
        \end{align*}
        for some $C > 0.$ The same arguments show that this bound remains valid for all $x \in B_2'$.
        This completes the proof of the lemma.
\end{proof}
\begin{corollary}\label{tau_V_estimate}
        There exist a positive constant $C$ and a compact set $D$ such that, uniformly in 
        $x\in\R_+^2\setminus D$ and in $l\ge 1$,
        \begin{align*}
                \P\bigl(\tau_x > l\bigr) \leqslant \frac{C V(x)}{l^{\frac{p}{2}}}.
        \end{align*}
\end{corollary}
\begin{proof}
        According to Theorem 3 in ~\cite{DW24} there exist positive constants $C_1$ and $R$ such that
        \begin{align}
        \label{eq:tau-u-bound}
                \P\bigl(\tau_x > l\bigr) \leqslant C_1\frac{u(x + Rx_0)}{l^\frac{p}{2}},\quad x\in\R_+^2
        \end{align}
        where again $x_0=(1,1)$.
        According to Lemma \ref{small_V},
        \begin{align*}
                \delta u(x)\leqslant V(x),\quad x\in R_+^2\setminus D.
        \end{align*}
        Moreover, for all $x \in Rx_0 + \R_+^2$ there exists $C_2=C_2(R)$ such that
        \begin{align*}
                u(x + Rx_0)\leqslant C_2 u(x).
        \end{align*}
        Combining that estimates we get the desired inequality.
\end{proof}
\begin{lemma}\label{PinDelta}
        Set $\Delta = [0, 1]^2$. For every $x\in\R_+^2$ there exists
        a constant $C(x)$ such that 
        \begin{align*}
                &\P\bigl(x + S(n)\in z + \Delta, \tau_x > n\bigr) \leqslant \frac{C(x)}
                {n^{\frac{p}{2} + 1}},\quad  z\in\R_+^2.
        \end{align*}
        Let $D$ be a compact set, for which Corollary~\ref{tau_V_estimate} holds. Then there exists a constant $C$ such that  
        \begin{align*}
                &\P\bigl(x + S(n)\in z + \Delta, \tau_x > n\bigr) \leqslant C\frac{V(x)}
                {n^{\frac{p}{2} + 1}},\quad x\in R_+^2\setminus D,\ z\in\R_+^2.
        \end{align*}
\end{lemma}
\begin{proof}
        A lattice version of the first statement has been proven in \cite{DW15}. The proof of the general case goes along the same line. 
       
        We first notice that
        \begin{align*}
                &\P\bigl(x + S(n)\in z + \Delta, \tau_x > n\bigr)\\
                &= \int\limits_{z\in\R_+^2}\P\bigl(x + S(n/2)\in d y,
                \tau_x > n/2\bigr)
                \P\bigl(y + S(n/2)\in z + \Delta, \tau_y > n/2\bigr)                
                \\
                &\leqslant\int\limits_{z\in\R_+^2}\P\bigl(x + S(n/2)\in d z, \tau_x > n/2\bigr)
                \P\bigl(z + S(n/2)\in y + \Delta\bigr)
                \\
                &\leqslant\int\limits_{z\in\R_+^2}\P\bigl(x + S(n/2)\in d z, \tau_x > n/2\bigr)
                \P\left(\left|S(n/2) - (y - z + \frac{\sqrt{2}}{2})\right|\leqslant 1\right).
        \end{align*}
        According to Theorem 6.2 of ~\cite{Esseen68},
        \begin{align*}
                \sup\limits_{x\in\R^2}\P\bigl(|S(n) - x|\leqslant 1\bigr)\leqslant \frac{C}{n}.
        \end{align*}
        Thus
        \begin{align*}
                \P\bigl(x + S(n)\in z + \Delta, \tau_x > n\bigr) 
                &\leqslant \frac{C}{n}\int\limits_{z\in\R_+^2}\P\bigl(x + S(n/2)\in d z,\tau_x > n/2\bigr)
                \\
                &\leqslant \frac{C}{n}\P\bigl(\tau_x > n/2\bigr).
        \end{align*}
        Combining this bound with Corollary~\ref{tau_V_estimate}, we get the second claim. To obtain the first inequality, it suffices to notice that \eqref{eq:tau-tail} implies that 
        $\P\bigl(\tau_x > n/2\bigr)\le C(x)n^{-p/2}$.
\end{proof}
\begin{corollary}
\label{cor:bounded_sets}
        Let $K$ be an arbitrary bounded set in $\R_+^2$ 
        For every $x\in\R_+^2$ one has 
        \begin{align*}
                \P\bigl(x + S(n)\in K,\tau_x > n\bigr)\leqslant 
                C(x)\frac{({\rm diam}(K)+1)^2 }{n^{\frac{p}{2} + 1}}.
        \end{align*}
        Furthermore, for all $x\in\R_+^2\setminus D$ one has
        \begin{align*}
                \P\bigl(x + S(n)\in K,\tau_x > n\bigr)\leqslant C\frac{({\rm diam}(K)+1)^2 V(x)}{n^{\frac{p}{2} + 1}}.
        \end{align*}
\end{corollary}
\begin{proof}
        As $K$ is a bounded set there exist $z_1, \dots, z_k$ such that $K\subset\bigcup\limits_{i = 1}^k 
        z_i + \Delta.$ Thus,
        \begin{align*}
                \P\bigl(x + S(n)\in K, \tau_x > n\bigr)\leqslant\sum\limits_{i = 1}^k
                \P\bigl(x + S(n)\in z_i + \Delta,\tau_x > n \bigr).
        \end{align*}
        Applying Lemma~\ref{PinDelta} and noting that the number of summands is bounded above by $C({\rm diam}(K)+1)^2$, we obtain the desired estimates.
\end{proof}

\subsection{Inequalities for one-dimensional walks.}
In this subsection we consider a one-dimensional random walk 
$T(n)=Y(1)+Y(2)+\ldots+Y(n)$, $n\ge1$ with i.i.d. increments 
$\{Y(n)\}$.
We shall always assume that $\E Y(1)=0$ and $\E Y^2(1) < \infty.$
\begin{lemma}\label{tail_sigma_estimation}
        For any stopping time $\sigma$ and for every $r>1$ we have
        \begin{align*}
                \P\bigl(T(n) \geqslant t, \sigma > n\bigr) \leqslant  \P\bigl(Y(1) \geqslant t/r\bigr)
                \sum\limits_{i = 1}^n \P\bigl(\sigma > i - 1\bigr) + 
                e^r\left(\frac{n r\E Y^2(1)}{t^2}\right)^r.
        \end{align*}
\end{lemma}
\begin{proof}
Set $\overline{Y}(n):=\max_{k\le n}Y(k)$.
Then

\begin{align*}
        \P\bigl(T(n) \geqslant t, \sigma > n\bigr)\leqslant\P\bigl(\overline{Y}(n)\ge t/r, \sigma > n\bigr)  
        + \P\bigl(T(n) \geqslant t, \overline{Y}(n)< t/r\bigr).
\end{align*}

According to Theorem 4 in Fuk and Nagaev~\cite{FukNagaev},
\begin{align*}
        \P\bigl(T(n) \geqslant t, \overline{Y}(n)< t/r\bigr)
        \leqslant 
        e^r\left(\frac{n r\E Y^2(1)}{t^2}\right)^r.
\end{align*}

Consider the  stopping time $\eta := \inf\{j\: : Y(j) \ge  t/r\}$. Then
\begin{align*}
\P\bigl(\overline{Y}(n)\ge t/r, \sigma > n\bigr) &= \P\bigl(\eta \leqslant n, 
\sigma > n\bigr)\leqslant \P\bigl(\eta\leqslant n, \eta < \sigma\bigr)\\
&=\sum\limits_{i = 1}^n \P\bigl(\eta = i, \sigma > i\bigr)\leqslant 
\sum\limits_{i = 1}^n \P\bigl(Y(i) > t, \sigma > i - 1\bigr). 
\end{align*}

Since $\sigma$ is a stopping time, the events $\{\sigma > i - 1\}$ and $\{Y(i>t)\}$ are independent. This completes the proof of the lemma.
\end{proof}

\begin{lemma}\label{lem:expect-tail}
        Fix some $p\ge1$ and assume that $\E\big|Y(1)\big|^q < \infty$ for some $q\geqslant p\vee 2$. Then, for any stopping time 
        $\sigma$ and for every $r>\frac{p-1}{2}\vee 1$ we have
\begin{align*}
\E\Bigl[(T(n))^{p - 1}; T(n) > t, \sigma > n\Bigr]
        &\leqslant r^q\frac{q}{q+1-p}\E|Y(1)|^q t^{p - 1 - q}\sum\limits_{i = 1}^n \P\bigl(\sigma > i - 1\bigr)
        \\
        &\hspace{1cm}+ (er\E Y^2(1))^r n^r\frac{2r}{2r+1-p}
          t^{p-1-2r}.
\end{align*}
\end{lemma}
\begin{proof}
Integrating by parts, we obtain
\begin{align}
\label{eq:fkm}
\nonumber
&\E\Bigl[(T(n))^{p - 1}; T(n) > t, \sigma > n\Bigr]\\
\nonumber
&\hspace{1cm}=\E\Bigl[{\rm I}_{\{\sigma > n\}}(T(n))^{p - 1}, T(n) > t\Bigr]
\nonumber
=\E\Bigl[\bigl({\rm I}_{\{\sigma > n\}}T(n)\bigr)^{p - 1}, T(n) > t\Bigr]\\
&\hspace{1cm}=t^{p - 1}\P\bigl({\rm I}_{\{\sigma > n\}}T(n) > t\bigr) + (p - 1)\int\limits_t^\infty s^{p - 2}
        \P\bigl({\rm I}_{\{\sigma > n\}}T(n) > s\bigr)ds.
\end{align}
Combining now Lemma \ref{tail_sigma_estimation} with the Markov inequality, we obtain
\begin{align*}
\P\bigl({\rm I}_{\{\sigma > n\}}T(n) > s\bigr)
&=\P\bigl(T(n) > s,\sigma>n\bigr)\\
&\le r^q\E|Y(1)|^q s^{-q}\sum_{i=1}^n\P(\sigma>i-1)
)+ (er\E Y^2(1))^r n^rs^{-2r}.
\end{align*}
Applying this to the right hand side of \eqref{eq:fkm}, we get 
\begin{align*}
&\E\Bigl[(T(n))^{p - 1}; T(n) > t, \sigma > n\Bigr]\\
&\hspace{1cm} 
\le r^q\E|Y(1)|^q\sum_{i=1}^n\P(\sigma>i-1)
\left(t^{p-1-q}+(p-1)\int_t^\infty s^{p-2-q}ds\right)\\
&\hspace{2cm}
+(er\E Y^2(1))^r n^r
\left(t^{p-1-2r}+(p-1)\int_t^\infty s^{p-2-2r}ds\right).
\end{align*}
Evaluating the integrals, we get the desired estimate.
\end{proof}
\subsection{Proof of Proposition~\ref{prop:repulsion}.}
Due to the Borel-Cantelli lemma, it suffices to show that 
$$
\sum_{n=1}^\infty\wP_x(x_1+S_1(n)\le\log^2 n)<\infty 
$$
and
$$
\sum_{n=1}^\infty\wP_x(x_2+S_2(n)\le\log^2 n)<\infty.
$$
Furthermore, by the symmetry, it suffices to prove one of that relations.
Thus, we shall concentrate on the first co-ordinate of the walk and shall bound every probability $\wP_x(x_1+S_1(n)\le\log^2 n)$ separately.
Since the co-ordinates are not independent we shall split this probability into two parts according to the value of $x_2+S_2(n)$.
\begin{lemma}
\label{lem:S2-small}
There exists a constant $C(x)$ such that, for all $1\le k_1\le k_2$ and all $n\ge1$,
$$
\wP_x(x_1+S_1(n)\le k_1,\,x_2+S_2(n)\le k_2)\le 
C(x)\frac{k_1^3k_2^p}{n^{(p+3)/2}}.
$$
\end{lemma}
\begin{proof}
By the definition of $\wP_x$,
\begin{align*}
&\wP_x(x_1+S_1(n)\le k_1,\,x_2+S_2(n)\le k_2)\\
&\hspace{1cm}
=\frac{1}{V(x)}\int_{\{z\in\R_+^2: z_i\le k_i\}}
V(z)\P(x+S(n)\in dz;\tau_x>n).
\end{align*}
Since $V$ is increasing in every co-ordinate,
\begin{align*}
&\wP_x(x_1+S_1(n)\le k_1,\,x_2+S_2(n)\le k_2)\\
&\hspace{1cm}
\le \frac{1}{V(x)}\sum_{\{v\in\Z_+^2: v_i< k_i\}}
V(v_1+1,v_2+1)\P(x+S(n)\in v+\Delta;\tau_x>n).
\end{align*}
In \cite{DW24} it has been proven that $V(x)\le Cu(x+Rx_0)$. This implies that 
\begin{align}
\label{eq:V-bound}
V(x)\le C(|x|^{p-1}+1)({\rm dist}(x,\partial\R_+^2)+1).
\end{align}
Consequently, 
\begin{align}
\label{eq:small-S2.1}
\nonumber
&\wP_x(x_1+S_1(n)\le k_1,\,x_2+S_2(n)\le k_2)\\
&\hspace{1cm}
\le \frac{C}{V(x)}k_1k_2^{p-1}\sum_{\{v\in\Z_+^2: v_i< k_i\}}
\P(x+S(n)\in v+\Delta;\tau_x>n).
\end{align}
Thus, we are left to estimate the probabilities $\P(x+S(n)\in v+\Delta;\tau_x>n)$.
By the Markov property,
\begin{align*}
&\P(x+S(n)\in v+\Delta;\tau_x>n)\\
&\hspace{1cm}
=\int_{\R_+^2}\P(x+S(n/2)\in dz;\tau_x>n/2)\P(z+S(n/2)\in v+\Delta;\tau_z>n/2).
\end{align*}
We now invert the time in the second probability term. Set
$$
X^*(j)=X(n-j+1),\quad j\le n/2
$$
and
$$
S^*(k)=\sum_{j=1}^k X^*(j),\quad k\le n/2.
$$
Then one has 
$$
\{z+S(n/2)\in v+\Delta;\tau_z>n/2\}
\subseteq
\{v+(1,1)+S^*(n/2)\in z+\Delta,T^*_v>n/2\},
$$
where 
$$
T_v=\inf\{k\ge1: v+(1,1)+S^*(k)\in\R_+\times\R\}.
$$
This implies that
\begin{align*}
&\P(x+S(n)\in v+\Delta;\tau_x>n)\\
&\le \sum_{w\in\Z_+^2}\P(x+S(n/2)\in w+\Delta,\tau_x>n/2)
\P(v+(1,1)+S^*(n/2)\in z+2\Delta,T^*_v>n/2).
\end{align*}
Applying now Lemma~\ref{PinDelta}, we get 
\begin{align*}
\P(x+S(n)\in v+\Delta;\tau_x>n)
\le \frac{C(x)}{n^{p/2+1}}\P(T^*_v>n/2).
\end{align*}
Noting that $T^*_v=\inf\{k\ge 1: v_1+1+S^*_1(k)\le0\}$
and applying Lemma 3 from \cite{DW16}, we have $\P(T^*_v>n/2)\le C(v_1+1)n^{-1/2}$.
Consequently, for all $v$ with $v_1\le k_1$,
$$
\P(x+S(n)\in v+\Delta;\tau_x>n)\le C(x)k_1n^{-(p+3)/2}.
$$
Plugging this bound into \eqref{eq:small-S2.1}, we get the desired estimate.
\end{proof}
Choosing $k_1=\log^2 n$ and $k_2=n^{1/2+\delta}$, we get 
\begin{corollary}
\label{cor:small-S2}
For every $n\ge 1$ we have 
$$
\wP_x(x_1+S_1(n)\le\log^2 n,x_2+S_2(n)\le n^{1/2+\delta})
\le C(x)\frac{n^{p\delta}\log^6n}{n^{3/2}}.
$$
In particular, the series $\sum_{n=1}^\infty \wP_x(x_1+S_1(n)\le\log^2 n,x_2+S_2(n)\le n^{1/2+\delta})$ is finite for every $\delta<1/(2p)$.
\end{corollary}
Thus, it remains to bound $\wP_x(x_1+S_1(n)\le\log^2 n,x_2+S_2(n)> n^{1/2+\delta})$.
This step is the most technical part in our approach. The main technical difficulty stems from the fact that we want to keep the minimal moment assumption $\E|X(1)|^p<\infty$. To illustrate that we first provide an estimate which is valid under a bit stronger moment condition but its proof is much simpler.
\begin{lemma}
\label{lem:S2-large.simple}
Assume that $\E|X(1)|^{p\vee2+1}<\infty$. Then there exists $\delta<\frac{1}{2p}$ such that 
$$
\sum_{n=1}^\infty
\wP_x\left(x_1+S_1(n)\le\log^2 n,x_2+S_2(n)> n^{1/2+\delta}\right)
<\infty.
$$
\end{lemma}
\begin{proof}
Using \eqref{eq:V-bound}, we conclude that 
\begin{align}
\label{eq:S2-large.s1}
\nonumber
&\wP_x\left(x_1+S_1(n)\le\log^2 n,x_2+S_2(n)> n^{1/2+\delta}\right)\\
\nonumber
&\hspace{1cm}
=\frac{1}{V(x)}\int_{\{z:z_1\le \log^2n,z_2>n^{1/2+\delta}\}}
V(z)\P(x+S(n)\in dz;\tau_x>n)\\
\nonumber
&\hspace{1cm}
\le \frac{C\log^2n}{V(x)}\E[|x_2+S_2(n)|^{p-1};x_2+S_2(n)>n^{1/2+\delta},\tau_x>n]\\
&\hspace{1cm}
\le C(x)\log^2n
\E\left[(S_2(n))^{p-1};S_2(n)>\frac{1}{2}n^{1/2+\delta},\tau_x>n\right].
\end{align}
Applying Lemma~\ref{lem:expect-tail}, we have
\begin{align}
\label{eq:S2-large.s2}
\nonumber
&\E\left[(S_2(n))^{p-1};S_2(n)>\frac{1}{2}n^{1/2+\delta},\tau_x>n\right]\\
&\hspace{1cm}
\le C_1(p,q,r)n^{(1/2+\delta)(p-1-q)}\sum_{k=0}^{n-1}\P(\tau_x>k)
+C_2(p,q,r)n^r(n^{1/2+\delta})^{p-1-2r},
\end{align}
where $q=p\vee2+1$. The second summand on the right hand side of \eqref{eq:S2-large.s2} equals 
$C_2(p,q,r)n^{(1/2+\delta)(p-1)-2r\delta}$, which is summable for all $r$ sufficiently large. Furthermore, the additional factor $\log^2n$ from \eqref{eq:S2-large.s1} does not affect the summability. To deal with the first summand on the right hand side of \eqref{eq:S2-large.s2} we notice
that \eqref{eq:tau-tail} yields
\begin{align}
\label{eq:tau-expectation}
\sum_{k=0}^{n-1}\P(\tau_x>n)
\le C(x)
\left\{
\begin{array}{ll}
1, &p>2,\\
\log n, &p=2,\\
n^{1-p/2}, &p\in(1,2).
\end{array}
\right.
\end{align}
If $p\ge2$ then $q=p+1$ and 
$$
n^{(1/2+\delta)(p-1-q)}\sum_{k=0}^{n-1}\P(\tau_x>k)
\le C(x)\log n\cdot n^{-1-2\delta}.
$$
If $p<2$ then $q=3$ and 
$$
n^{(1/2+\delta)(p-1-q)}\sum_{k=0}^{n-1}\P(\tau_x>k)
\le C(x)n^{(1/2+\delta)(p-4)}n^{1-p/2}
\le C(x)n^{-1-2\delta}.
$$
Thus, first terms on the right hand side of \eqref{eq:S2-large.s2}
are dominated by $\frac{\log n}{ n^{1+2\delta}}$. This yields the desired summability.
\end{proof}
\begin{remark}
If we take in the proof above $q=p\vee2$, which corresponds to our moment conditions in Theorem~\ref{thm:coexistence} and in Proposition~\ref{prop:repulsion}, then we get an upper bound of order $\frac{\log n}{n^{1/2+\delta}}$. Therefore, $n^{-1/2}$ is missing to have again a summable upper bound. This factor is lost in the first inequality in \eqref{eq:S2-large.s1}, where we drop the condition $x_1+S_1(n)\le \log^2 n$. In the following lemma we shall do a more careful analysis which will allow to remove the extra moment condition in Lemma~\ref{lem:S2-large.simple}. 
\hfill$\diamond$
\end{remark}
\begin{lemma}
\label{lem:S2-large}
Under the conditions of Proposition~\ref{prop:repulsion} one has
\begin{align*}
 &\wP_x\left(x_1+S_1(n)\le k_1,x_2+S_2(n)> k_2\right)\\
&\hspace{1cm}\le C(p,r)C(x)\left[\frac{k_1^2\log n}{\sqrt{n}k_2}+
\frac{k_1^2}{\sqrt{n}}n^rk_2^{p-1-2r}\right]
\end{align*}
for all $k_1\le \sqrt{n}\le k_2$.
\end{lemma}
\begin{proof}
Similarly to \eqref{eq:S2-large.s1} we have
\begin{align}
\label{eq:S2-large.s3}
\nonumber
&\wP_x\left(x_1+S_1(n)\le k_1,x_2+S_2(n)> k_2\right)\\
\nonumber
&\hspace{1cm}
=\frac{1}{V(x)}\int_{\{z:z_1\le k_1,z_2>k_2\}}
V(z)\P(x+S(n)\in dz;\tau_x>n)\\
&\hspace{1cm}
\le \frac{Ck_1}{V(x)}\E[|x_2+S_2(n)|^{p-1};x_2+S_2(n)>k_2,x_1+S_1(n)\le k_1,\tau_x>n].
\end{align}
We first notice that 
\begin{align}
\label{eq:S2-large.s4}
\nonumber
&\E[|x_2+S_2(n)|^{p-1};x_2+S_2(n)>k_2,x_1+S_1(n)\le k_1,\tau_x>n]\\
\nonumber
&\le \E\left[|x_2+S_2(n)|^{p-1};x_2+S_2(n)>k_2,x_1+S_1(n)\le k_1,x_2+S_2\left(\frac{n}{2}\right)\ge\frac{k_2}{2},\tau_x>n\right]\\
\nonumber 
&\hspace{2mm} 
+\E\left[|x_2+S_2(n)|^{p-1};x_2+S_2(n)>k_2,x_1+S_1(n)\le k_1,\left|S_2(n)-S_2\left(\frac{n}{2}\right)\right|\ge\frac{k_2}{2},\tau_x>n\right]\\
\nonumber
&\le \E\left[\left|x_2+S_2\left(\frac{n}{2}\right)\right|^{p-1};x_1+S_1(n)\le k_1,x_2+S_2\left(\frac{n}{2}\right)\ge\frac{k_2}{2},\tau_x>n\right]\\ 
&\hspace{2mm} 
+\E\left[\left|S_2(n)-S_2\left(\frac{n}{2}\right)\right|^{p-1};x_1+S_1(n)\le k_1,
\left|S_2(n)-S_2\left(\frac{n}{2}\right)\right|\ge\frac{k_2}{2},\tau_x>n\right].
\end{align}
By the Markov property at time $\frac{n}{2}$,
\begin{align*}
&\E\left[\left|x_2+S_2\left(\frac{n}{2}\right)\right|^{p-1};x_1+S_1(n)\le k_1,x_2+S_2\left(\frac{n}{2}\right)\ge\frac{k_2}{2},\tau_x>n\right]\\
&=\int_{\{y\in\R_+^2:y_2\ge k_2/2\}}
y_2^{p-1}\P\left(x+S\left(\frac{n}{2}\right)\in dy,\tau_x>n/2\right)
\P\left(y_1+S\left(\frac{n}{2}\right)\le k_1,\tau_y>n/2\right)\\
&\le\int_{\{y\in\R_+^2:y_2\ge k_2/2\}}
y_2^{p-1}\P\left(x+S\left(\frac{n}{2}\right)\in dy,\tau_x>n/2\right)
\P\left(y_1+S\left(\frac{n}{2}\right)\in(0,k_1]\right).
\end{align*}
Applying now the standard bound for the concentration function, we get 
$$
\P\left(y_1+S\left(\frac{n}{2}\right)\in(0,k_1]\right)
\le C\frac{k_1}{\sqrt{n}}
$$
and, consequently,
\begin{align*}
&\E\left[\left|x_2+S_2\left(\frac{n}{2}\right)\right|^{p-1};x_1+S_1(n)\le k_1,x_2+S_2\left(\frac{n}{2}\right)\ge\frac{k_2}{2},\tau_x>n\right]\\
&\le C\frac{k_1}{\sqrt{n}}\int_{\{y\in\R_+^2:y_2\ge k_2/2\}}
y_2^{p-1}\P\left(x+S\left(\frac{n}{2}\right)\in dy,\tau_x>n/2\right)\\
&\le C\frac{k_1}{\sqrt{n}}\E\left[
\left|x_2+S_2\left(\frac{n}{2}\right)\right|^{p-1};
x_2+S_2\left(\frac{n}{2}\right)\ge k_2/2,\tau_x>n/2\right].
\end{align*}
Using Lemma~\ref{lem:expect-tail}, we conclude that, for all $k_2\ge 3x_2$,
\begin{align*}
&\E\left[
\left|x_2+S_2\left(\frac{n}{2}\right)\right|^{p-1};
x_2+S_2\left(\frac{n}{2}\right)\ge k_2/2,\tau_x>n/2\right]\\
&\hspace{1cm}
\le C_1(p,q,r)k_2^{p-1-q}\sum_{i=0}^{n-1}\P(\tau_x>i)
+C_2(p,q,r)n^rk_2^{p-1-2r}.
\end{align*}
As a result we have
\begin{align}
\label{eq:S2-large.s5}
\nonumber
&\E\left[\left|x_2+S_2\left(\frac{n}{2}\right)\right|^{p-1};x_1+S_1(n)\le k_1,x_2+S_2\left(\frac{n}{2}\right)\ge\frac{k_2}{2},\tau_x>n\right]\\
&\hspace{1cm}\le C(p,q,r)\frac{k_1}{\sqrt{n}}
\left(k_2^{p-1-q}\sum_{i=0}^{n-1}\P(\tau_x>i)
+n^rk_2^{p-1-2r}\right).
\end{align}
For the second expectation on the right hand side of \eqref{eq:S2-large.s4}
we have
\begin{align*}
&\E\left[\left|S_2(n)-S_2\left(\frac{n}{2}\right)\right|^{p-1};x_1+S_1(n)\le k_1,
\left|S_2(n)-S_2\left(\frac{n}{2}\right)\right|\ge\frac{k_2}{2},\tau_x>n\right]\\
&=\int_{y\in\R_+^2}\P(x+S(n/2)\in dy;\tau_x>n/2)\\
&\hspace{2cm}\times
\E\left[\left(S_2\left(\frac{n}{2}\right)\right)^{p-1};
S_2\left(\frac{n}{2}\right)>k_2/2,y_1+S_1(n)\le k_1,\tau_y>n/2
\right].
\end{align*}
We next notice that, uniformly in $y$,
\begin{align*}
&\left\{y_1+S_1(n/2)\le k_1,\tau_y>n/2\right\}\\
&\hspace{0.5cm}\subseteq 
\left\{y_1+S_1(n/2)\le k_1,y_1+S_1(k)>0\text{ for all }k\le n/2\right\}\\
&\hspace{0.5cm}\subseteq 
\left\{y_1+S_1(n/2)\le k_1,y_1+S_1(n/2)-(S_1(n/2)-S_1(k))>0\text{ for all }k\le n/2\right\}\\
&\hspace{0.5cm}\subseteq
\left\{k_1+S_1^*(k)>0\text{ for all }k\le n/2\right\},
\end{align*}
where $S^*(k)$ is the random walk introduced in the proof of Lemma~\ref{lem:S2-small}. This implies that 
\begin{align*}
&\E\left[\left(S_2\left(\frac{n}{2}\right)\right)^{p-1};
S_2\left(\frac{n}{2}\right)>k_2/2,y_1+S_1(n)\le k_1,\tau_y>n/2
\right]\\
&\le 
\E\left[\left(-S^*_2\left(\frac{n}{2}\right)\right)^{p-1};
-S^*_2\left(\frac{n}{2}\right)>k_2/2,T^*_{(k_1,0)}>n/2
\right]
\end{align*}
and, consequently,
\begin{align*}
&\E\left[\left|S_2(n)-S_2\left(\frac{n}{2}\right)\right|^{p-1};x_1+S_1(n)\le k_1,
\left|S_2(n)-S_2\left(\frac{n}{2}\right)\right|\ge\frac{k_2}{2},\tau_x>n\right]\\
&\hspace{1cm}\le \P(\tau_x>n/2)
\E\left[\left(-S^*_2\left(\frac{n}{2}\right)\right)^{p-1};
-S^*_2\left(\frac{n}{2}\right)>k_2/2,T^*_{(k_1,0)}>n/2
\right].
\end{align*} 
Combining Lemma~\ref{lem:expect-tail} and the inequality
$\P(T_{(k_1,0)}>j)\le C\frac{k_1}{\sqrt{j}}$, we obtain 
\begin{align*}
&\E\left[\left(-S^*_2\left(\frac{n}{2}\right)\right)^{p-1};
-S^*_2\left(\frac{n}{2}\right)>k_2/2,T^*_{(k_1,0)}>n/2
\right]\\
&\hspace{1cm}\le C_1(p,q,r)k_2^{p-1-q}\sum_{j=0}^{n-1}\P(T_{(k_1,0)}>j)
+C_2(p,q,r)n^rk_2^{p-1-2r}\\
&\hspace{1cm}\le C(p,q,r)\left(k_2^{p-1-q}k_1 n^{1/2}+n^rk_2^{p-1-2r}\right).
\end{align*}
Recalling that $\P(\tau_x>n/2)\le C(x)n^{-p/2}$, we finally get 
\begin{align}
\label{eq:S2-large.s6} 
\nonumber
&\E\left[\left|S_2(n)-S_2\left(\frac{n}{2}\right)\right|^{p-1};x_1+S_1(n)\le k_1,
\left|S_2(n)-S_2\left(\frac{n}{2}\right)\right|\ge\frac{k_2}{2},\tau_x>n\right]\\
&\hspace{1cm}
\le C(x)n^{-p/2}\left(k_2^{p-1-q}k_1 n^{1/2}+n^rk_2^{p-1-2r}\right).
\end{align}

Consider first the case $p\ge2$. In this case we have $q=p$.
Plugging \eqref{eq:S2-large.s5} and \eqref{eq:S2-large.s6} with
$q=p$ into \eqref{eq:S2-large.s4} and using \eqref{eq:tau-expectation}, we get
\begin{align*}
&\E[|x_2+S_2(n)|^{p-1};x_2+S_2(n)>k_2,x_1+S_1(n)\le k_1,\tau_x>n]\\
&\hspace{1cm}\le C(p,r)C(x)\left[\frac{k_1\log{n}}{\sqrt{n}k_2}
+\frac{k_1\sqrt{n}}{n^{p/2}k_2}
+\frac{k_1}{\sqrt{n}}n^rk_2^{p-1-2r}
+\frac{1}{n^{p/2}}n^rk_2^{p-1-2r}\right]\\
&\hspace{1cm}\le 2C(p,r)C(x)\left[\frac{k_1\log n}{\sqrt{n}k_2}+
\frac{k_1}{\sqrt{n}}n^rk_2^{p-1-2r}\right].
\end{align*}
In the case $p<2$ we have $q=2$. This leads to the bound 
\begin{align*}
&\E[|x_2+S_2(n)|^{p-1};x_2+S_2(n)>k_2,x_1+S_1(n)\le k_1,\tau_x>n]\\
&\hspace{1cm}\le C(p,r)C(x)\left[\frac{k_1k_2^{p-3}}{n^{(p-1)/2}}+
k_1k_2^{p-1-2r}n^{r-1/2}\right]\\
&\hspace{1cm}\le C(p,r)C(x)\left[\frac{k_1\log n}{\sqrt{n}k_2}+
\frac{k_1}{\sqrt{n}}n^rk_2^{p-1-2r}\right],
\end{align*}
in the last step we have used the assumption $k_2>\sqrt{n}$.
This completes the proof of the lemma.
\end{proof}
\begin{corollary}
\label{cor:S22-large}
We have
\begin{align*}
&\wP_x\left(x_1+S_1(n)\le \log^2 n,x_2+S_2(n)>n^{1/2+\delta}\right)\\
&\hspace{1cm}=O\left(\frac{\log^5 n}{n^{1+\delta}}+
\log^4 n\frac{n^{(p-1)/2}}{n^{\delta(2r-p+1)}}\right).
\end{align*}
Taking $r$ sufficiently large we also conclude that
$$
\sum_{n=1}^\infty \wP_x\left(x_1+S_1(n)\le \log^2 n,x_2+S_2(n)> n^{1/2+\delta}\right)<\infty.
$$
\end{corollary}
Proposition~\ref{prop:repulsion} is immediate from Corollaries~\ref{cor:small-S2} and \ref{cor:S22-large}.

\section{Proof of Theorem~\ref{thm:coexistence}.}
If all offspring distributions are geometric then one has the following expression for conditional non-extinction probabilities, see \cite{GKbook},
$$
\P_z(Z_i(n)>0|\mathbf{\Pi})
=1-\left(\frac{\sum_{k=1}^n e^{-S_i(k)}}{1+\sum_{k=1}^n e^{-S_i(k)}}\right)^{z_i},\quad i=1,2.
$$
Conditionally on the environment $\mathbf{\Pi}$, the random variables $Z_1(n)$ and $Z_2(n)$ are independent. This implies that  
\begin{align}
\label{eq:non-ext.prob.cond}
\nonumber
&\P(Z_1(n)>0,Z_2(n)>0|\mathbf{\Pi})\\
&\hspace{0.5cm}=\left(1-\left(\frac{\sum_{k=1}^n e^{-S_1(k)}}{1+\sum_{k=1}^n e^{-S_1(k)}}\right)^{z_1}\right)
\left(1-\left(\frac{\sum_{k=1}^n e^{-S_2(k)}}{1+\sum_{k=1}^n e^{-S_2(k)}}\right)^{z_2}\right)
\end{align}
and, consequently,
\begin{align}
\label{eq:non-ext.prob}
\nonumber
&\P(Z_1(n)>0,Z_2(n)>0)\\
&\hspace{0.5cm}=\E\left[\left(1-\left(\frac{\sum_{k=1}^n e^{-S_1(k)}}{1+\sum_{k=1}^n e^{-S_1(k)}}\right)^{z_1}\right)
\left(1-\left(\frac{\sum_{k=1}^n e^{-S_2(k)}}{1+\sum_{k=1}^n e^{-S_2(k)}}\right)^{z_2}\right)\right].
\end{align}
Thus, we have to determine the asymptotic behaviour of the exponential functional of the associated random walk $S(n)$ given on the right hand side of \eqref{eq:non-ext.prob}. As in the one-dimensional case, the crucial tool in the proof is the following lemma.
\begin{lemma}\label{cond_limit}
Let $\{Y(k)\}_{k\ge1}$ be a sequence of random variables adapted to the natural filtration $\{\sigma(S(1),S(2),\ldots,S(k))\}_{k\ge1}$. We shall assume that this sequence is uniformly bounded by a constant $k_Y$, that is,
$|Y(k)|\le k_Y$ for all $k$.
\begin{itemize}
        \item[(a)]If $x$ is such that $V(x)>0$ then
                \begin{align*}
                        \lim\limits_{n\to\infty} \E\Bigl[Y(k)\big|\tau_x > n\Bigr] = \wE_x\Bigl[Y(k)\Bigr] 
                \end{align*}
                for every fixed $k$.
        \item[(b)] Assume again that $x$ is such that $V(x)>0$. If there exists a random variables $Y(\infty)$ such that 
        $\lim\limits_{k\to\infty} Y(k) = Y(\infty)$ $\wP_x$-almost 
                surely. Then 
                \begin{align*}
                        \lim\limits_{n\to\infty} \E\Bigl[Y(n)\big|\tau_x > n\Bigr] = \wE_x\Bigl[Y(\infty)\Bigr]. 
                \end{align*}
        \end{itemize}
\end{lemma}
\begin{proof}
Recalling that the random variable $Y(k)$ is measurable with respect to $\sigma(S(1),S(2),\ldots,S(k))$ and applying the Markov property, we obtain
$$
\E[Y(k);\tau_x>n]
=\int_{\R_+^2}\E\left[Y(k){\rm I}_{\{\tau_x>k\}};x+S(k)\in dz\right]
\P(\tau_z>n-k).
$$
In view of \eqref{eq:tau-tail},
$$
n^{p/2}\P(\tau_z>n-k)\to V(z)
\quad\text{for every }z\in K.
$$
Furthermore, \eqref{eq:tau-u-bound} implies that we may apply Lebesgue's theorem on dominated convergence. As a result we have
\begin{align*}
&\E[Y(k)|\tau_x>n]\\
&\hspace{1cm}
=\frac{1}{\P(\tau_x>n)}\E[Y(k);\tau_x>n]\\
&\hspace{1cm}
=\frac{1}{n^{p/2}\P(\tau_x>n)}
\int_{\R_+^2}\E\left[Y(k){\rm I}_{\{\tau_x>k\}};x+S(k)\in dz\right]
n^{p/2}\P(\tau_z>n-k)\\
&\hspace{1cm}
\rightarrow\frac{1}{V(x)}\E[V(x+S(k))Y(k);\tau_x>k]=\wE_x[Y(k)]
\quad \text{as }n\to\infty.
\end{align*}
Thus, the first convergence is proven.

To prove the second claim we first consider the difference
$$
\E[Y(k)|\tau_x>\chi n]-\E[Y(n)|\tau_x>\chi n],
$$
where $\chi>1$ is a constant. Applying first the triangle inequality and using then the Markov property at time $n$, we have 
\begin{align*}
&\left|\E[Y(k)|\tau_x>\chi n]-\E[Y(n)|\tau_x>\chi n]\right|\\
&\hspace{5mm}
\le\frac{1}{\P(\tau_x>\chi n)}\E[|Y(k)-Y(n)|;\tau_x>\chi n]\\
&\hspace{5mm}
=\frac{1}{\P(\tau_x>\chi n)}\int_{\R_+^2}
\E[|Y(k)-Y(n)|{\rm I}_{\{\tau_x>n\}};x+S(n)\in dz]
\P(\tau_z>(\chi-1)n). 
\end{align*}
Let $D$ be a compact set for which Corollary~\ref{tau_V_estimate} holds. Then
\begin{align}
\label{eq:first-part-a}
\nonumber
&\int_{\R_+^2\setminus D}
\E[|Y(k)-Y(n)|{\rm I}_{\{\tau_x>n\}};x+S(n)\in dz]
\P(\tau_z>(\chi-1)n)\\
\nonumber
&\hspace{1cm}
\le \frac{C}{(\chi-1)^{p/2}n^{p/2}}
\int_{\R_+^2\setminus D}
\E[|Y(k)-Y(n)|V(x+S(n));\tau_x>n]\\
&\hspace{1cm}\le C\frac{V(x)}{(\chi-1)^{p/2}n^{p/2}}
\wE |Y(k)-Y(n)|.
\end{align}
Furthermore, recalling that random variables $\{Y(l)\}$ are uniformly bounded and applying Corollary~\ref{cor:bounded_sets}, we get 

\begin{align}
\label{eq:first-part-b}
\nonumber
&\int_{D}
\E[|Y(k)-Y(n)|{\rm I}_{\{\tau_x>n\}};x+S(n)\in dz]
\P(\tau_z>(\chi-1)n)\\
\nonumber
&\hspace{1cm}
\le 2k_Y\P(x+S(n)\in D;\tau_x>n)\\
&\hspace{1cm}
\le C(x,D)k_Y n^{-p/2-1}.
\end{align}
Combining \eqref{eq:first-part-a} and \eqref{eq:first-part-b},
and recalling the assumption that $Y(n)$ converges to $Y(\infty)$, we conclude that
\begin{align}
\label{eq:first-part}
\nonumber
&\limsup_{n\to\infty}
\left|\E[Y(k)|\tau_x>\chi n]-\E[Y(n)|\tau_x>\chi n]\right|\\
&\hspace{2cm}
\le C\left(\frac{\chi}{\chi-1}\right)^{p/2}
\wE_x|Y(k)-Y(\infty)|.
\end{align}

Using once again the assumption that $|Y(l)|\le k_Y$ for all $l$
and applying \eqref{eq:tau-tail}, we get 
\begin{align}
\label{eq:second-part}
\nonumber
&\limsup_{n\to\infty}
\left|\E[Y(n)|\tau_x>n]-\E[Y(n)|\tau_x>\chi n]\right|\\
\nonumber 
&\hspace{1cm}
\le\limsup_{n\to\infty} 2k_Y
\frac{\P(\tau_x>n)-\P(\tau_x>\chi n)}{\P(\tau_x>\chi n)}\\
&\hspace{1cm}
=2k_Y(\chi^{p/2}-1).
\end{align}
By the first part of the lemma,
\begin{align}
\label{eq:third-part}
\lim_{n\to\infty}\E[Y(k)|\tau_x>\chi n]=\wE_x[Y(k)]
\end{align}
for every fixed $k$. 
Combining \eqref{eq:first-part}, \eqref{eq:second-part}, \eqref{eq:third-part} with the triangle inequality 
\begin{align*}
&\left|\E[Y(n)|\tau_x>n]-\wE_x[Y(\infty)]\right|\\
&\hspace{5mm}
\le\left|\E[Y(n)|\tau_x>n]-\E[Y(n)|\tau_x>\chi n]\right|
+\left|\E[Y(n)|\tau_x>\chi n]-\E[Y(k)|\tau_x>\chi n]\right|\\
&\hspace{1cm}+
\left|\E[Y(k)|\tau_x>\chi n]-\wE_x[Y(k)]\right|
+\left|\wE_x[Y(k)]-\wE_x[Y(\infty)]\right|,
\end{align*}
we infer that 
\begin{align*}
&\limsup_{n\to\infty}
\left|\E[Y(n)|\tau_x>n]-\wE_x[Y(\infty)]\right|\\
&\hspace{3mm}\le
2k_Y(\chi^{p/2}-1)+
C\left(\frac{\chi}{\chi-1}\right)^{p/2}
\wE_x|Y(k)-Y(\infty)|
+\left|\wE_x[Y(k)]-\wE_x[Y(\infty)]\right|.
\end{align*}
Letting $k\to\infty$ and using the assumption that $Y(n)$ converges towards $Y(\infty)$, we have 
\begin{align*}
\limsup_{n\to\infty}
\left|\E[Y(n)|\tau_x>n]-\wE_x[Y(\infty)]\right|
\le
2k_Y(\chi^{p/2}-1).
\end{align*}
Letting now $\chi\downarrow 1$, we complete the proof of the second claim.
\end{proof}

        Later we shall apply this lemma to random variables on the right hand side of \eqref{eq:non-ext.prob}.
Define
\begin{align}
\label{eq:Y-def}
Y(n)=
\left(1-\left(\frac{\sum_{k=1}^n e^{-S_1(k)}}{1+\sum_{k=1}^n e^{-S_1(k)}}\right)^{z_1}\right)
\left(1-\left(\frac{\sum_{k=1}^n e^{-S_2(k)}}{1+\sum_{k=1}^n e^{-S_2(k)}}\right)^{z_2}\right),
\end{align}
        \begin{align*}
                m(n) := \min\bigl\{S_1(k), S_2(k),\:\: 0\leqslant k\leqslant n\bigr\}
        \end{align*}
        and consider the family of events
        \begin{align*}
                D_{n,R} := \bigl\{m(n)\in[-R - 1, -R]\bigr\},\quad n\ge 1,\ R>0.
        \end{align*}

        Noting that 
        \begin{align*}
                e^{-S_j(k)}\geqslant \exp\Bigl\{-m(n)\Bigr\},\quad j=1,2,
        \end{align*}
        we have $Y(n)\le e^{-R}$ and, consequently,
        \begin{align*}
                \E\Biggl[Y(n); D_{n, R}\Biggr]
                \leqslant e^{-R}\P\bigl(D_{n, R}\bigr).
        \end{align*}
        Furthermore we have 
        \begin{align*}
                D_{n, R}\subset\bigl\{m(n) > -R - 1\bigr\}.
        \end{align*}
        This implies that
        \begin{align*}
                \P\bigl(D_{n, R}\bigr)\leqslant\P\bigl(\tau_{(R + 1, R + 1)} > n\bigr). 
        \end{align*}
        Thus, applying Corollary~\ref{tau_V_estimate}, we obtain, for all sufficiently large $R$,
        \begin{align*}
                \P\bigl(D_{n, R}\bigr)\leqslant 
                \frac{C V\bigl(R + 1, R + 1\bigr)}{n^\frac{p}{2}}\leqslant\frac{CR^p}{n^\frac{p}{2}}.
        \end{align*}
        Consequently,
        \begin{align*}
                \E\left[Y(n); D_{n, R}\right]
                \leqslant CR^pe^{-R}n^{-p/2}.
        \end{align*}
        Let us fix some sufficiently large $R_0$. Then one has
        \begin{align}\label{left_tail}
        \nonumber
                \E\left[Y(n);
                m(n) \leqslant -R_0\right] 
                &=
                \sum\limits_{R = R_0}^\infty\E\left[Y(n); D_{n, R}\right]\\
                &\leqslant \frac{C}{n^\frac{p}{2}}\sum\limits_{R = R_0}^\infty R^p e^{-R}
                \leqslant\frac{C}{n^\frac{p}{2}}\int\limits_{R_0}^\infty z^p e^{-z}dz \leqslant \frac{C}
                {n^\frac{p}{2}}R_0^p e^{-R_0}.
        \end{align}
        Thus, we are left to determine the asymptotic behaviour of
        \begin{align*}
                \E\left[Y(n); m(n) > -R_0\right].                
        \end{align*}
        We first notice that  
        \begin{align*}
                \E\left[Y(n); m(n) > -R_0\right]
                &=
                \E\left[Y(n); \tau_{(R_0, R_0)} > n\right] 
                \\
                &
                = \E\left[Y(n)\big|\:\: 
                \tau_{(R_0, R_0)} > n\right]\P\bigl(\tau_{(R_0, R_0)} > n\bigr).
        \end{align*}
        
        As $\sum\limits_{k = 1}^n e^{-S_i(k)} \geqslant 0$ for $i = 1, 2$, the sequence $Y(n)$ is uniformly bounded. Moreover,
        $\{Y(n)\}$ is decreasing and bounded by $0$ from below. Therefore, the sequence $\{Y(n)\}$ converges
        $\wP_{(R_0, R_0)}$-a.s. to some $Y(\infty)$. Furthermore, Proposition~\ref{prop:repulsion} implies that 
        $\wP_{(R_0,R_0)}(Y(\infty)>0)=1$.

        Thus, according to Lemma \ref{cond_limit},
        \begin{align*}
                \lim\limits_{n\to\infty}\E\bigl[Y(n)\big|\:\:\tau_{(R_0, R_0) > n}\bigr] = \wE_{(R_0, R_0)}
                \bigl[Y(\infty)\bigr]>0. 
        \end{align*}
        Combining this with \eqref{eq:tau-tail}, we infer that, as $n\to\infty$,
        \begin{align}\label{right_tail}
                \E\left[Y(n); m(n) > -R_0\right]
                \sim \frac{\varkappa}{n^\frac{p}{2}}V\bigl(R_0, R_0\bigr)\wE_{(R_0, R_0)}
                \bigl[Y(\infty)\bigr].
        \end{align}
        Combining \eqref{left_tail} and \eqref{right_tail}, we obtain
        \begin{align}
        \label{eq:lower-bound} 
        \liminf\limits_{n\to\infty} n^\frac{p}{2}\P\bigl(Z_1(n) > 0, Z_2(n) > 0\bigr)>0
        \end{align}
and
        \begin{align}\label{boundary}
                \limsup\limits_{n\to\infty} n^\frac{p}{2}\P\bigl(Z_1(n) > 0, Z_2(n) > 0\bigr)\leqslant C_b<
                \infty.
        \end{align}
        Consider now the function
        \begin{align*}
                g(y) := V(y, y)\wE_{(y, y)}\bigl[Y(\infty)\bigr],
        \end{align*}
        for all $y$ sufficiently large $y$, say $y\ge y_0$. Relation \eqref{boundary} implies that $g(y)\in [0;C_b]$ for all $y\geqslant y_0.$
        According to Bolzano–Weierstrass theorem there exists a sequence $\{y_l\}_{l = 1}^\infty$ such that 
        $y_l\to\infty$ and $\lim\limits_{l\to\infty}g(y_l) = C_g.$

        According to \eqref{right_tail} with $R_0=y_l$,
        \begin{align*}
                \P\bigl(Z_1(n) > 0, Z_2(n) > 0\bigr) \geqslant \P\bigl(Z_1(n) > 0, Z_2(n) > 0,\:\: m(n) > -y_l\bigr)
                \sim\frac{\varkappa\cdot g(y_l)}{n^\frac{p}{2}},
        \end{align*}
        which implies that
        \begin{align*}
                \ \liminf\limits_{n\to\infty}n^\frac{p}{2}\P\bigl(Z_1(n) > 0, 
                Z_2(n) > 0\bigr)\ge \varkappa g(y_l).
        \end{align*}
        To get the corresponding upper bound we combine \eqref{right_tail} and \eqref{left_tail}:
        \begin{align*}
                &\limsup\limits_{n\to\infty}n^\frac{p}{2}\P\bigl(Z_1(n) > 0, Z_2(n) > 0\bigr) 
                \\
                &\hspace{1cm}= 
                \limsup\limits_{n\to\infty}n^\frac{p}{2}\cdot\Bigg(\P\bigl(Z_1(n) > 0, Z_2(n) > 0,\:\: m(n) \leqslant 
                -y_l\bigr) 
                \\
                &\hspace{2cm}
                + \P\bigl(Z_1(n) > 0, Z_2(n) > 0,\:\: m(n) > -y_l\bigr)\Biggr)\\
                &\hspace{1cm}
                \leqslant\varkappa g(y_l) + C_2y_l^p e^{-y_l}.
        \end{align*}
        Therefore,
        \begin{align*}
                \varkappa g(y_l)&\leqslant \liminf\limits_{n\to\infty}n^\frac{p}{2}\P\bigl(Z_1(n) > 0, 
                Z_2(n) > 0\bigr)
                \\
                &\leqslant
                \limsup\limits_{n\to\infty}n^\frac{p}{2}\P\bigl(Z_1(n) > 0, Z_2(n) > 0\bigr)
                \leqslant \varkappa g(y_l) + C_2y_l^p e^{-y_l}.
        \end{align*}
        Leting $l \to \infty$ we have
        \begin{align}
        \label{eq:R-limit}
        \lim_{R_0\to\infty} V(R_0,R_0)\wE_{(R_0,R_0)}[Y(\infty)]=C_g
        \end{align}
        and  
        \begin{align}
        \label{eq:limit}
                \lim\limits_{n\to\infty}n^\frac{p}{2}\P\bigl(Z_1(n) > 0, Z_2(n) > 0\bigr) = \varkappa C_g.
        \end{align}
Noting that \eqref{eq:lower-bound} implies the positivity of the limit in \eqref{eq:limit}, we finish the proof of Theorem~\ref{thm:coexistence}.
\section{Proof of Theorem~\ref{thm:cond.limit}}
The proof below uses the standard for branching processes in random environments strategy. We start by showing that the trajectories of 
$z^{(n)}$ are close the trajectories of the rescaled random walk 
$S(k)$.
\begin{lemma}
\label{lem:z-s}
For every $\varepsilon>0$ we have 
$$
\P_z\left(\max_{k\le n}|\log Z_i(k)-S_i(k)|\ge\varepsilon\sqrt{n},
Z_i(n)>0
\big|\mathbf{\Pi}\right)\le Cz_ine^{-\varepsilon\sqrt{n}}
\quad n\ge1.
$$
\end{lemma}
\begin{proof}
Since $\frac{Z_i(k)}{e^{S_i(k)}}$ is a martingale, by the Doob inequality,
\begin{align}
\label{eq:z-s.1}
\nonumber
\P_z\left(\max_{k\le n}\left(\log Z_i(k)-S_i(k)\right)\ge\varepsilon\sqrt{n}
\big|\mathbf{\Pi}\right)
&\le\P_z\left(\max_{k\le n}\frac{Z_i(k)}{e^{S_i(k)}}\ge e^{\varepsilon\sqrt{n}}\right)\\
&\le 
z_ie^{-\varepsilon\sqrt{n}}.
\end{align}
To estimate deviations of $\log Z_i(k)$ to the left from $S_i(k)$ we notice that 
\begin{align}
\label{eq:z-s.2}
\nonumber
&\P_z\left(\min_{k\le n}\left(\log Z_i(k)-S_i(k)\right)\le-\varepsilon\sqrt{n},Z_i(n)>0\big|\mathbf{\Pi}\right)\\
&\hspace{2cm}\le \sum_{k=1}^n
\P_z\left(1\le Z_i(n)\le e^{S_i(k)-\varepsilon\sqrt{n}}
\big|\mathbf{\Pi}\right).
\end{align}
Our assumption that all offspring distributions are geometric allows us to write down an explicit formula for generating functions of $Z_i(k)$. Let $F_{i,k}(s)$ denote the conditional generating function 
of $Z_i(k)$ starting from one ancestor, that is,
$$
F_{i,k}(s)=\E\left[s^{Z_i(k)}\big|Z_i(0)=1,\mathbf{\Pi}\right].
$$
One has then the following equality:
$$
\frac{1}{1-F_{i,k}(s)}
=\sum_{j=0}^{k-1}e^{-S_i(j)}+\frac{e^{-S_i(k)}}{1-s}.
$$
This implies that 
$$
1-F_{i,k}(s)=\frac{1-s}{\sum_{j=0}^{k}e^{-S_i(j)}-s\sum_{j=0}^{k-1}e^{-S_i(j)}}
$$
and, consequently,
\begin{align*}
F_{i,k}(s)-F_{i,k}(0)
&=(1-F_{i,k}(0))-(1-F_{i,k}(s))\\
&=\frac{se^{-S_i(k)}}{\sum_{j=0}^{k}e^{-S_i(j)}\left(\sum_{j=0}^{k}e^{-S_i(j)}-s\sum_{j=0}^{k-1}e^{-S_i(j)}\right)}\\
&\le \frac{e^{-S_i(k)}}{\left(\sum_{j=0}^{k}e^{-S_i(j)}-s\sum_{j=0}^{k-1}e^{-S_i(j)}\right)}
\end{align*}
for all $s\in(0,1]$. 
Setting here
$$
s=s_{k,n}=1-\frac{e^{-S_i(k)+\varepsilon\sqrt{n}}}{2},
$$
we get 
\begin{align}
\label{eq:F-bound}
F_{i,k}(s_{k,n})-F_{i,k}(0)
\le \frac{e^{-S_i(k)}}
{e^{-S_i(k)}+\frac{e^{-S_i(k)+\varepsilon\sqrt{n}}}{2}}
\le 2 e^{-\varepsilon\sqrt{n}}.
\end{align}
We next notice that 
\begin{align*}
&\P_z\left(1\le Z_i(k)\le e^{S_i(k)-\varepsilon\sqrt{n}}
\big|\mathbf{\Pi}\right)\\
&\hspace{1cm}\le \left(s_{k,n}\right)^{-e^{S_i(k)-\varepsilon\sqrt{n}}}
\sum_{1\le m\le e^{S_i(k)-\varepsilon\sqrt{n}}}
s_{k,n}^m\P_z\left(Z_i(k)=m\big|\mathbf{\Pi}\right)\\
&\hspace{1cm}
\le \left(s_{k,n}\right)^{-e^{S_i(k)-\varepsilon\sqrt{n}}}
\left((F_{i,k}(s_{k,n}))^{z_i}-(F_{i,k}(0))^{z_i}\right)\\
&\hspace{1cm}
\le z_i\sup_{u\ge1}\left(1-\frac{1}{2u}\right)^{-u}
\left(F_{i,k}(s_{k,n})-F_{i,k}(0)\right).
\end{align*}
Combining this with \eqref{eq:F-bound}, we have 
\begin{align*}
 \P_z\left(1\le Z_i(k)\le e^{S_i(k)-\varepsilon\sqrt{n}}
\big|\mathbf{\Pi}\right)
\le (2e)z_ie^{-\varepsilon\sqrt{n}}.
\end{align*}
Plugging this into \eqref{eq:z-s.2}, we get the desired bound for lower deviations. This completes the proof of the lemma.
\end{proof}
Define
$$
s^{(n)}(t)
=\left(\frac{S_1(nt)}{\sqrt{n}},\frac{S_2(nt)}{\sqrt{n}}\right),
\quad t\in[0,1].
$$
Then, Lemma~\ref{lem:z-s} yields the following bound.
\begin{corollary}
\label{cor:z-s}
For every $\varepsilon>0$ one has 
$$
\P_z\left(\sup_{t\le1}|z^{(n)}(t)-s^{(n)}(t)|\ge\varepsilon
\Big|E(n)\right)\to0.
$$
\end{corollary}
\begin{proof}
Combining Lemma~\ref{lem:z-s} and Theorem~\ref{thm:coexistence}, we have 
$$
\P_z\left(\sup_{t\le1}|z^{(n)}(t)-s^{(n)}(t)|\ge\varepsilon
\Big|E(n)\right)\le Cn^{\theta+1}e^{-\varepsilon\sqrt{n}}.
$$
This implies the desired convergence.
\end{proof}

Let $f$ a uniformly continuous, bounded functional on $D[0,1]$.
Then, for every $\delta>0$ there exists $\varepsilon>0$ such that 
$|f(h_1)-f(h_2)|<\delta$ for all $h_1,h_2$ such that 
$\sup_{t\le 1}|h_1(t)-h_2(t)|<\varepsilon$. This implies that 
\begin{align*}
&\left|\E_z[f(z^{(n)})|E(n)]-\E_z[f(s^{(n)})|E(n)]\right| \\
&\hspace{1cm}\le \delta
+2\|f\|\P_z\left(\sup_{t\le1}|z^{(n)}(t)-s^{(n)}(t)|\ge\varepsilon
\Big|E(n)\right).
\end{align*}
Applying Corollary~\ref{cor:z-s} and letting then $\delta\to0$, we conclude that 
$$
\lim_{n\to\infty}\left|\E_z[f(z^{(n)})|E(n)]-\E_z[f(s^{(n)})|E(n)]\right|=0. 
$$
Thus, it remains to determine the limit of $\E_z[f(s^{(n)})|E(n)]$.
Conditioning on the environment, we obtain 
\begin{align}
\label{eq:lt1}
\E_z[f(s^{(n)});E(n)]
=\E[\E_z[f(s^{(n)}){\rm I}_{\{E(n)\}}|\mathbf{\Pi}]
=\E[f(s^{(n)})Y(n)],
\end{align}
where $Y(n)$ is the random variable defined in \eqref{eq:Y-def}.
As in the proof of Theorem~\ref{thm:coexistence} we split this expected value into two parts:
$$
\E[f(s^{(n)})Y(n)]
=\E[f(s^{(n)})Y(n);m(n)\le -R_0]+\E[f(s^{(n)})Y(n); m(n)>-R_0].
$$
Recalling that $f$ is bounded and applying \eqref{left_tail}, we have 
\begin{align}
\label{eq:lt2}
\nonumber
\E[f(s^{(n)})Y(n);m(n)\le -R_0]
&\le \|f\|\E[Y(n);m(n)\le -R_0]\\
&\le C\|f\|R_0^{p/2}e^{-R_0}n^{-p/2}.
\end{align}
Furthermore, for every fixed $k$,
\begin{align}
\label{eq:lt3}
\nonumber
&\E[f(s^{(n)})Y(n); m(n)>-R_0]\\
\nonumber
&\hspace{5mm}=\E[f(s^{(n)})Y(n); \tau_{(R_0,R_0)}>n]\\
&\hspace{5mm}=\E[f(s^{(n)})Y(k); \tau_{(R_0,R_0)}>n]
+\E[f(s^{(n)})(Y(n)-Y(k)); \tau_{(R_0,R_0)}>n].
\end{align}
Recalling that the sequence $\{Y(n)\}$ is monotone decreasing and applying Lemma~\ref{cond_limit}, we conclude that 
\begin{align}
\label{eq:lt4}
\nonumber
&\left|\E[f(s^{(n)})(Y(n)-Y(k)); \tau_{(R_0,R_0)}>n]\right|\\
\nonumber
&\hspace{1cm}
\le \|f\|\E[(Y(k)-Y(n)); \tau_{(R_0,R_0)}>n]
\\
\nonumber
&\hspace{1cm}
=\|f\|\P(\tau_{(R_0,R_0)}>n)\E[(Y(k)-Y(n))| \tau_{(R_0,R_0)}>n]\\
&\hspace{1cm}
\le C\|f\|\frac{V(R_0,R_0)}{n^{p/2}}
\wE_{(R_0,R_0)}[Y(k)-Y(\infty)].
\end{align}
To deal with $\E[f(s^{(n)})Y(k); \tau_{(R_0,R_0)}>n]$ we define 
$$
s_k^{(n)}(t)=\left\{ 
\begin{array}{ll}
\frac{S(k)}{\sqrt{n}}, &t\le\frac{k}{n},\\
s^{(n)}(t), &t\ge\frac{k}{n}.
\end{array}
\right.
$$
Our moment assumptions imply that 
\begin{align*}
\P\left(\max_{t\le 1}|s_k^{(n)}(t)-s^{(n)}(t)|\ge \varepsilon\right) 
\le\P\left(2\max_{j\le k}|S(j)|>\varepsilon\sqrt{n}\right)
=o(n^{-p/2}).
\end{align*}
This implies that 
\begin{align}
\label{eq:lt5}
\E[f(s^{(n)})Y(k); \tau_{(R_0,R_0)}>n]
=\E[f(s_k^{(n)})Y(k); \tau_{(R_0,R_0)}>n]+o(n^{-p/2}).
\end{align}
By the Markov property,
\begin{align*}
&\E[f(s_k^{(n)})Y(k); \tau_{(R_0,R_0)}>n]\\
&=\int_{\R_+^2}\E[Y(k);(R_0,R_0)+S(k)\in dx,\tau_{(R_0,R_0)}>k]
\E[f(s_{k,x}^{(n)}); \tau_x>n-k],
\end{align*}
where 
$$
s_{k,x}^{(n)}(t)=\left\{ 
\begin{array}{ll}
\frac{x}{\sqrt{n}}, &t\le\frac{k}{n},\\
\frac{x+S(nt-k)}{\sqrt{n}}, &t\ge\frac{k}{n}.
\end{array}
\right.
$$
Applying Theorem 2 from \cite{DuW20} as it was described in Section~2 and taking into account \eqref{eq:tau-tail}, we conclude that 
\begin{align*}
 \E[f(s_{k,x}^{(n)}); \tau_x>n-k]\sim\varkappa V(x)\E[f(M_\varrho)]
 n^{-p/2}
\end{align*}
for every $x\in\R_+^2$. Noting additionally that Corollary~\ref{tau_V_estimate} provides an integrable majorant for 
$n^{p/2}\E[f(s_{k,x}^{(n)}); \tau_x>n-k]$, we apply the Lebesgue theorem to get 
$$
n^{p/2}\E[f(s_k^{(n)})Y(k); \tau_{(R_0,R_0)}>n]
\to\varkappa\E[f(M_\varrho)] V(R_0,R_0)\wE_{(R_0,R_0)}[Y(k)].
$$
Combining this with \eqref{eq:lt3}---\eqref{eq:lt5}, we have 
\begin{align*}
&\limsup_{n\to\infty}
\Big|n^{p/2}\E[f(s^{(n)})Y(n); m(n)>-R_0]-\varkappa\E[f(M_\varrho)] V(R_0,R_0)\wE_{(R_0,R_0)}[Y(k)]\Big|\\
&\hspace{3cm}
\le C\|f\|V(R_0,R_0)\wE_{(R_0,R_0)}[Y(k)-Y(\infty)].
\end{align*}
Letting $k\to\infty$ and recalling that $Y(k)$ converges to $Y(\infty)$, we get 
$$
\lim_{n\to\infty}
n^{p/2}\E[f(s^{(n)})Y(n); m(n)>-R_0]=\varkappa\E[f(M_\varrho)] V(R_0,R_0)\wE_{(R_0,R_0)}[Y(\infty)].
$$
From this relation and from bound \eqref{eq:lt2} we obtain 
$$
\lim_{n\to\infty}n^{p/2}\E[f(s^{(n)})Y(n)]
=\varkappa\E[f(M_\varrho)]
\lim_{R_0\to\infty}V(R_0,R_0)\wE_{(R_0,R_0)}[Y(\infty)].
$$
Combining this relation with \eqref{eq:lt1} and taking into account \eqref{eq:R-limit}, we conclude that 
$$
\lim_{n\to\infty}n^{p/2}\E[f(s^{(n)});E(n)]
=\varkappa C_g \E[f(M_\varrho)].
$$
Applying now \eqref{eq:limit}, we finally get 
$$
\lim_{n\to\infty}\E[f(s^{(n)})|E(n)]=\E[f(M_\varrho)]
$$
for every uniformly continuous bounded functional $f$. Thus, the proof of Theorem~\ref{thm:cond.limit} is complete.


\begin{thebibliography}{99}


 \bibitem{DW15} Denisov, D. and Wachtel, V.
    \newblock Random walks in cones.
    \newblock \emph{Ann. Probab.}, {\bf 43}: 992--1044, 2015.

 \bibitem{DW16} Denisov, D. and Wachtel, V.
    \newblock An exact asymptotics for the moment of crossing a curved boundary by an asymptotically stable random walk
    \newblock \emph{Theory Probab. Appl.}, {\bf 60}: 481--500, 2016.
    
 \bibitem{DW24} Denisov, D. and Wachtel, V.
    \newblock Random walks in cones revisited.
    \newblock \emph{Ann. Inst. Henri Poincaré Probab. Stat.}, {\bf 60}: 126--166, 2024.   
    
 \bibitem{DuW20} Duraj, J. and Wachtel, V.
    \newblock Invariance principles for random walks in cones.
    \newblock\emph{Stochastic Process. Appl.}, {\bf  130}: 3920--3942, 2020.

 \bibitem{Esseen68}  Esseen, C.G.
    \newblock On the concentration function of a sum of independent random variables.
    \newblock\emph{Z. Warsch. Verw. Gebiete}, {\bf 9}: 290--308, 1968.

 \bibitem{FukNagaev} Fuk, D.X. and Nagaev, S.V.
    \newblock Probability inequalities for sums of independent random variables.
    \newblock\emph{Theory Probab. Appl.}, {\bf 16}:643--660, 1971.

 \bibitem{GK02} Geiger, J. and Kersting, G.
    \newblock The survival probability of a critical branching process in random environment.
    \newblock \emph{Theory Probab. Appl.}, {\bf 45}:518--526, 2002.

 \bibitem{GKbook} Kersting, G. and Vatutin, V.
    \newblock\emph{Discrete time branching processes in random environment.}
    \newblock John Wiley \& Sons, 2017 

 \bibitem{PPP18}   Le Page, E., Peigne, M. and Pham, C.
    \newblock The survival probability of a critical multi-type branching process in i.i.d. random environment.
    \newblock \emph{Ann. Probab.}, {\bf 46}: 2946--2972, 2018.
    
 \bibitem{VL13} Vatutin, V.A. and Liu, Q.
    \newblock Critical branching process with two types of particles evolving in asynchronous random environments.
    \newblock \emph{Theory Probab. Appl.}, {\bf 57}: 279-305, 2013.
    
 
\end{thebibliography}
\end{document}